\documentclass[]{article}
\usepackage{cite}
\usepackage{amsmath,amssymb,amsfonts,amsthm}
\usepackage{algorithmic}
\usepackage{graphicx}
\usepackage{textcomp}
\usepackage{url}
\usepackage{appendix}

\newtheorem{exmp}{Example}[section]
\newtheorem{theorem}{Theorem}[section]

\newtheorem{proposition}[theorem]{Proposition}
\newtheorem{remark}{Remark}[section]

\newtheorem{assum}{Assumption}[section]

\usepackage{tikz}
\usetikzlibrary{arrows.meta,positioning}

\newcommand*\samethanks[1][\value{footnote}]{\footnotemark[#1]}

\begin{document}
\title{Parallel Model Predictive Control for Deterministic Systems\thanks{This work was partially supported by the Swedish Research Council Distinguished Professor Grant 2017-01078, the Knut and Alice Wallenberg Foundation Wallenberg Scholar Grant, and the Swedish Strategic Research Foundation CLAS Grant RIT17-0046. It was also supported by the Swiss National Science Foundation under NCCR Automation (grant agreement 51NF40\_180545).}}
\author{Yuchao Li\thanks{Y. Li was with the division of Decision and Control Systems, KTH Royal Institute of Technology, Sweden, when the majority of this work was conducted. He is now with the School of Computing and Augmented Intelligence, Arizona State University, 
        {\tt\small yuchaoli@asu.edu}.} \and Aren Karapetyan\thanks{A. Karapetyan, N. Schmid, J. Lygeros are with the Automatic Control Laboratory, Swiss Federal Institute of Technology in Zürich, Switzerland,
        {\tt\small akarapetyan,nikschmid,jlygeros@ethz.ch}.}%
\and Niklas Schmid\samethanks \and John Lygeros\samethanks \and Karl H. Johansson\thanks{K. H. Johansson, and J. M\aa rtensson are with the division of Decision and Control Systems, KTH Royal Institute of Technology, Sweden. They are also affiliated with Digital Futures and Integrated Transport Research Lab, {\tt\small kallej,jonas1@kth.se}.} \and Jonas M\aa rtensson\samethanks} 
%
%

\maketitle

\begin{abstract}
In this note, we consider infinite horizon optimal control problems with deterministic systems. Since exact solutions to these problems are often intractable, we propose a parallel model predictive control (MPC) method that provides an approximate solution. Our method computes multiple lookahead minimization problems at each time, where each minimization may involve a different number of lookahead steps, and terminal cost and constraint. {The policy computed via parallel MPC applies the first control of the lookahead minimization with the lowest cost. We show that the proposed method can harnesses the power of multiple computing units. Moreover, we prove that the policy computed via parallel MPC has better performance guarantee than that computed via the single lookahead minimization involved in parallel MPC.}
\end{abstract}

\section{Introduction}\label{sec:introduction}

In this work, we consider discrete-time optimal control problems involving deterministic systems 
\begin{equation}
\label{eq:dynamics}
    x_{k+1}=f(x_k,u_k),\quad k=0,\,1,\,\dots,
\end{equation}
where $x_k$ and $u_k$ are state and control at stage $k$ belonging to state and control spaces $X$ and $U$, respectively, and $f$ maps $X\times U$ to $X$. The control $u_k$ must be chosen from a nonempty constraint set $U(x_k)\subset U$ that may depend on $x_k$. The cost of applying $u_k$ at state $x_k$ is denoted by $g(x_k,u_k)$, and is assumed to be nonnegative:
\begin{equation}
\label{eq:cost}
    0\leq g(x_k,u_k)\leq\infty,\quad x_k\in X,\,u_k\in U(x_k).
\end{equation}
By allowing an infinite value of $g(x,u)$ we can implicitly introduce state and control constraints: a pair $(x,u)$ is infeasible if $g(x,u)=\infty$. We consider \emph{stationary} feedback policies $\mu$, which are functions mapping $X$ to $U$ that satisfy $\mu(x)\in U(x)$ for all $x$.

The \emph{cost function} of a policy $\mu$, denoted by $J_\mu$, maps $X$ to $[0,\infty]$, and is defined at any initial state $x_0 \in X$ as
$$J_\mu(x_0)= \sum_{k=0}^\infty g(x_k,\mu(x_k)),$$
where $x_{k+1}=f(x_k,\mu(x_k))$, $k=0,\,1,\,\dots$. The optimal cost function $J^*$ is defined pointwise as
\begin{equation*}
    J^*(x_0)=\inf_{\substack{u_k\in U(x_k),\ k=0,1,\ldots\\ x_{k+1}=f(x_k,u_k),\ k=0,1,\ldots}}\sum_{k=0}^\infty  g(x_k,u_k).
\end{equation*}
A stationary policy $\mu^*$ is called optimal if
\begin{equation*}
    J_{\mu^*}(x)=J^*(x),\quad \forall x\in X.
\end{equation*}
Optimal policies satisfy
\begin{equation}
\label{eq:mu_star}
    \mu^*(x)\in \arg\min_{u\in U(x)} \big\{g(x,u)+J^*\big(f(x,u)\big)\big\},\quad \forall x\in X,
\end{equation}
{assuming the minimum in \eqref{eq:mu_star} can be attained for all $x$}.

In the context of dynamic programming (DP), the goal is to obtain an optimal policy $\mu^*$ and optimal cost function $J^*$. However, computing $\mu^*$ and $J^*$ is often intractable. To obtain an approximate solution, one may use some known function {$J$} in place of $J^*$ and {compute a policy $\tilde \mu$ via a minimization calculation similar to \eqref{eq:mu_star} at each time}. If {$J$} resembles $J^*$ in some way, we may expect the obtained policy $\tilde{\mu}$ close to optimal, in the sense that $J_{\tilde{\mu}}\approx J^*$. This idea constitutes the essence of a broad class of suboptimal schemes.

A popular method of this kind is model predictive control (MPC), which defines a policy through online computation similar to \eqref{eq:mu_star}.\footnote{{We refer to computations as ``online" if they are executed in real time, and as ``off-line" if they are performed in advance of online operations.}} At state $x_k$, the MPC method considers the state sequence $x_k\dots,x_{k+\ell}$ generated under the control sequence $u_k,\dots,u_{k+\ell-1}$, where the number $\ell$ is referred to as (the length of) the \emph{lookahead steps}. The desirability of different sequences are compared using the corresponding sums of the stage costs, plus a function defined on $x_{k+\ell}$. In addition, the state $x_{k+\ell}$ may also be restricted within a set. The function and the constraint associated with $x_{k+\ell}$ are called \emph{terminal cost} and \emph{terminal constraint}, respectively. Collectively, they aim to approximate the optimal cost $J^*$. {The resulting optimization problem is simpler than computing $J^*$, and may be possible to solve online.} Suppose the minimum of the problem is attained at $\Tilde{u}_k,\dots,\Tilde{u}_{k+\ell-1}$, MPC policy $\tilde{\mu}$ is defined by setting $\tilde{\mu}(x_k)=\Tilde{u}_k$. When $\ell$ is sufficiently large, and the terminal cost and constraint have certain properties, the resulting MPC policy can be close to optimal. 

In contrast to the MPC method that involves a single minimization problem, a different suboptimal scheme, known as parallel rollout, solves simultaneously multiple copies of \eqref{eq:mu_star}. Each copy of the minimization problem in parallel rollout uses a different cost function $J_\mu$ in place of $J^*$, corresponding to a different policy $\mu$. The method then uses the control corresponding to the problem with the lowest cost. Parallel rollout was first proposed to address finite horizon problems, and has been popular in contexts involving discrete state and/or control elements.

Taking inspiration from parallel rollout, we propose a new MPC scheme, which we call \emph{parallel MPC}. It requires solving multiple MPC lookahead minimizations at each time, where each minimization can involve different lookahead steps, and terminal cost and/or constraint. Compared with the MPC scheme noted earlier, the policy $\tilde{\mu}$ defined by parallel MPC has better performance guarantees. Let us consider the following scalar example to illustrate the scheme.

\begin{exmp}\label{eg:scalar_intro} (Parallel MPC for a Scalar Problem) Let $X=\Re$, $U(x)=U=[-1,1]$ for all $x$, and $f(x,u)=2x+u$. Assume that $g(x,u)=x^2+u^2$ if $|x|\leq 2$, and $g(x,u)=\infty$ otherwise. For this problem, an example of parallel MPC can be described as follows: at state $x_k$, solve the following two optimization problems: 
\begin{align*}
    &\min_{u_k,u_{k+1}\in U}g(x_k,u_k)+g(x_{k+1},u_{k+1})\;\;\mathrm{s.\,t.}\,x_{k+2}=0,\\
&\min_{u_k\in U}g(x_k,u_k)+(13/3)x_{k+1}^2\;\;\mathrm{s.\,t.}\,|x_{k+1}|\leq 2/3,
\end{align*}
{with additional constraints $x_{k+1}=2x_k+u_k$, and $x_{k+2}=2x_{k+1}+u_{k+1}$; cf. Eq.~\eqref{eq:dynamics}}. {The procedure through which the terminal costs and constraints are computed will be explained in Example~\ref{eg:scalar_intro_central}.} Then the optimal values and the controls $\Tilde{u}_k$ that attain the minimum in the two problems are computed, and the control that corresponds to the minimum of the two optimal values is applied by parallel MPC.
\end{exmp}

The contributions of this work are as follows:
\begin{itemize}
    \item[(a)] We introduce a parallel MPC method that involves solving multiple minimization problems with different lookahead horizons, and terminal costs and/or constraints.
    \item[(b)] {We show that parallel MPC can harnesses the power of multiple computing units.}
    \item[(c)] {We prove that the policy defined by parallel MPC has a better performance guarantee than that defined by the single lookahead minimization involved in parallel MPC.}
    
\end{itemize}

Note that our method and analysis also apply to the problems with an additional discount factor $\alpha\in(0,1)$ on future stage costs. We restrict our attention to the problems without a discount factor as this is the primary context where MPC methods are used.

{The paper is organized as follows.} In Section~\ref{sec:background} we provide background and references, which place in context our results in relation to the literature. In Section~\ref{sec:main}, we describe our MPC algorithm and its variants, and we provide analysis. In Section~\ref{sec:exmp}, we demonstrate our method through computational examples.

\section{Background}\label{sec:background}
The study of infinite horizon optimal control with deterministic systems and nonnegative cost has a long and rich history. Exact solutions have been derived for some special types of problems, see, e.g., {\cite{bellman1958routing,kalman1960contributions,li2024exact,li2025semilinear}}. For a rigorous analysis that applies generally, see \cite{bertsekas2015value}, where the properties of $J^*$, as well as algorithmic behaviors for exact solution methods, are addressed. 

Despite these extensive research efforts, many deterministic optimal control problems remain intractable. As a result, one often resorts to approximate solution methods. For problems with continuous state and control spaces under constraints, MPC has had abundant applications with wide success. The central themes of MPC analysis are maintaining the \emph{feasibility} of the $\ell$-step lookahead minimization problems upon reaching new states, and the \emph{stability} of the resulting closed-loop system under the obtained policy. The former question is typically addressed by introducing a terminal constraint with an invariance property \cite{bertsekas1971control}. The terminal cost is often designed to fulfill a condition, known as the \emph{Lyapunov inequality}, to provide the stability guarantees \cite{mayne2000constrained}. Since the early work \cite{keerthi1988optimal}, there have been many MPC variants with these desired properties. Among them, examples involving multiple terminal constraints and/or costs include \cite{marruedo2002enlarging,fitri2020nonlinear,li2021data}. {Yet these methods lack generality or flexibility compared with parallel MPC introduced here.} Our method is also different from the MPC schemes reported in \cite{jerez2011parallel,jiang2020parallel}, which focus on solving efficiently a single lookahead minimization via parallel computation. In contrast, our methods involves multiple lookahead minimization problems solved in parallel, and our goal is to improve the performance of the MPC policy measured by its cost function.

As for the problems with discrete state and/or control elements, parallel rollout has been shown to be effective. The method was first introduced in \cite{bertsekas1997rollout} (under the name \emph{rollout with multiple heuristics}) to address combinatorial problems. Since then, parallel rollout and its variants have been extended to address a variety of other problems, see, e.g., \cite{bertsekas1999rollout,chang2004parallel,antunes2014rollout}. For a comprehensive discussion, see the recent monograph \cite{bertsekas2020rollout} and the references quoted therein. The analysis associated with parallel rollout focuses on ensuring a performance improvement property, which implies that the cost function of the policy $\tilde{\mu}$ computed via parallel rollout is no more than all the cost functions $J_\mu$ used to approximate $J^*$. This property of parallel rollout can be proven by using the result \cite[Prop. 6.1.1]{bertsekas2017dynamic} that holds more generally for functions other than the cost functions $J_\mu$. Parallel MPC leverages this generality by allowing more flexible choices of the function to approximate $J^*$ and the numbers of lookahead steps.  

Although MPC and parallel rollout are proposed in different contexts, these methods can be unified through a conceptual framework that is based on DP; see the monograph \cite{bertsekas2022lessons} and the surveys \cite{bertsekas2022newton,bertsekas2024model}. In particular, it is shown that both MPC and parallel rollout are one step of the Newton's method applied to compute the optimal cost function $J^*$, which is responsible for the wide success of these methods. Moreover, the analysis of MPC and rollout can be unified by using a condition that connects to the Lyapunov inequality, as is done in \cite[Sections 6.1 and 6.5]{bertsekas2017dynamic}. We will use this framework to introduce and to analyze the parallel MPC method.

\subsection*{Technical Preliminaries}
We apply the abstract DP model given in \cite{bertsekas2018abstract} in order to concisely state our method and to streamline our proofs. In particular, we denote as $\mathcal{E}^+(X)$ the set of all functions $J:X\mapsto[0,\infty]$. A mapping that plays a key role in our development is the \emph{Bellman operator} $T:\mathcal{E}^+(X)\mapsto\mathcal{E}^+(X)$, defined pointwise as
\begin{equation}
    \label{eq:bellman_op}
    (TJ)(x)=\inf_{u\in U(x)} \big\{g(x,u)+J\big(f(x,u)\big)\big\}.
\end{equation}
This operator is well-posed in view of the nonnegativity \eqref{eq:cost} of the stage cost. We denote as $T^\ell$ the $\ell$-fold composition of $T$, with the convention that $T^0J=J$. From \eqref{eq:bellman_op}, it can be seen that if $J(x)\leq J'(x)$ for all $x\in X$, we have that $(TJ)(x)\leq (TJ')(x)$ for all $x\in X$; we refer to this as the \emph{monotonicity property} of $T$. We are particularly interested in a subset of $\mathcal{E}^+(X)$, which we call the \emph{L-region} (L stands for Lyapunov), defined as 
\begin{equation}
    \label{eq:rod}
    \mathcal{L}(X)=\{J\in \mathcal{E}^+(X)\,|\,(TJ)(x)\leq J(x),\,\forall x\in X\}.
\end{equation}
Since the function $J\in \mathcal{E}^+(X)$ can take the value infinity, it can be used to encode an invariance property, as we will show in Section~\ref{sec:main}. 

Throughout our analysis, we make frequent reference to the following monotonicity property: For all nonnegative functions $J$ and $J'$ that map $X$ to $[0,\infty]$, if $J(x)\leq J'(x)$ for all $x\in X$, then for all $x\in X,\,u\in U(x)$, we have
    \begin{equation}
    \label{eq:monotone}
        g(x,u)+J\big(f(x,u)\big)\leq g(x,u)+J'\big(f(x,u)\big).
    \end{equation}
In addition, we also use the following classical result, which holds well beyond the scope of our study, and can be found in \cite[Prop.~4.1.2, Prop.~4.1.4(a)]{bertsekas2012dynamic}.
\begin{proposition}\label{prop:classic}
Under the nonnegativity condition \eqref{eq:cost}, the following hold. 
\begin{itemize}
    \item[(a)] For all stationary policies $\mu$ we have
    \begin{equation*}
        J_\mu(x)=g\big(x,\mu(x)\big)+J_\mu\big(f\big(x,\mu(x)\big)\big),\ \ \forall x\in X.
    \end{equation*}
    \item[(b)] {For all stationary policies $\mu$, if a nonnegative function $J:X\to [0,\infty]$ satisfies
    \begin{equation*}
        g\big(x,\mu(x)\big)+J\big(f\big(x,\mu(x)\big)\big)\leq J(x),\ \ \forall x\in X,
    \end{equation*}
    then $J_\mu(x)\leq J(x)$ for all $x\in X$.}
\end{itemize}
\end{proposition}

\section{Main Results}\label{sec:main}
Let us formally state our parallel MPC algorithm and analyze its properties. We impose the following assumption, which will allow us to replace $\inf$ by $\min$ thereafter. \emph{The assumption remains valid throughout this {work} and thus is omitted in the subsequent theoretical statements.}
\begin{assum}\label{assum:standing}
The cost nonnegativity condition \eqref{eq:cost} holds. Moreover, for all functions $J\in \mathcal{E}^+(X)$ that we consider, the minimum in the following optimization
\begin{equation*}
    \inf_{u\in U(x)} \big\{g(x,u)+J\big(f(x,u)\big)\big\}
\end{equation*}
is attained for all $x\in X$.
\end{assum}

{Note that our method and analysis is valid under Assumption~\ref{assum:standing}, regardless of the nature of the state and control spaces.} On the other hand, different types of state and control spaces, being continuous or discrete, may impact the validity of the assumption, and can determine the ways in which we solve the lookahead minimization problems involved in our method.

\subsection{Description of Parallel MPC}
Consider the optimal control problem with {state equation} \eqref{eq:dynamics} and stage cost \eqref{eq:cost}. We assume that we have access to $p$ computational units as well as a central unit. Given the current state $x_k$, computational unit $i=1,\dots,p$ solves the problem
\begin{subequations}
\label{eq:rollout}
    \begin{align}
	\min_{\{u_{k+j}\}_{j=0}^{\ell_i-1}}& \quad J_i(x_{k+\ell_i})+\sum_{j=0}^{\ell_i-1} g(x_{k+j},u_{k+j})\\
	\mathrm{s.\,t.} & \quad x_{k+j+1}=f(x_{k+j},u_{k+j}),\,j = 0,...,\ell_i-1,\label{eq:rollout_dynamics}\\
	& \quad  u_{k+j} \in U(x_{k+j}), \ j = 0,...,\ell_i-1, \label{eq:rollout_inputconstraints}
	\end{align}
\end{subequations} 
where $\ell_i$ is some positive integer, which can vary with $i$, and $J_i\in\mathcal{L}(X)$ is a function \emph{computed off-line in explicit form}.\footnote{{We have restricted our attention to the cases where the functions $J_i$ are known in explicit forms. Regarding the situation where the functions $J_i$ can only be evaluated online through simulation, as in parallel rollout given in \cite[Section~5.2]{bertsekas2020rollout}, our method and analysis to be introduced shortly applies in principle. However, in such a case, the lookahead minimization involved in our method can be solved only if the control space is discrete.}} The optimal value of the problem obtained by the $i$th unit is denoted by $\Tilde{J}_i(x_k)$, and the corresponding minimizing sequence is denoted by $(\Tilde{u}^i_k,\Tilde{u}^i_{k+1},\dots,\Tilde{u}^i_{k+\ell_i-1})$. The values $\Tilde{J}_i(x_k)$ and $\Tilde{u}_k^i$ are then sent to the central unit that computes
\begin{equation}
    \label{eq:index_unit}
    \tilde{i}\in\arg \min_{i=1,2,\dots,p}\Tilde{J}_i(x_k).
\end{equation}
The parallel MPC policy is then defined by setting $\tilde{\mu}(x_k)=\Tilde{u}^{\tilde{i}}_k$. If we employ the Bellman operator \eqref{eq:bellman_op}, the algorithm can be stated succinctly as computing 
\begin{equation}
    \label{eq:bellman_p_rollout}
    \tilde{J}(x_k)=\min_{i=1,2,\dots,p}(T^{\ell_i}J_i)(x_k),
\end{equation}
by the central unit, where the values $(T^{\ell_i}J_i)(x_k)$ are supplied by the $p$ computational units. Let us now connect the parallel MPC stated in Example~\ref{eg:scalar_intro} with its generic form given in \eqref{eq:rollout}.
\begin{exmp}\label{eg:scalar_intro_re} In the parallel MPC discussed in Example~\ref{eg:scalar_intro}, there are $2$ computing units. The lookahead minimization problems solved in the two units are:
\begin{equation}
\label{eq:scalar_parallel}
    \begin{aligned}
    &\text{unit 1: }\min_{u_k,u_{k+1}\in U}g(x_k,u_k)+g(x_{k+1},u_{k+1})\;\;\mathrm{s.\,t.}\,x_{k+2}=0,\\
    &\text{unit 2: }\min_{u_k\in U}g(x_k,u_k)+(13/3)x_{k+1}^2\;\;\mathrm{s.\,t.}\,|x_{k+1}|\leq 2/3,
\end{aligned}
\end{equation}
{with additional constraints $x_{k+1}=2x_k+u_k$, and $x_{k+2}=2x_{k+1}+u_{k+1}$}. Here $\ell_1=2$, and $\ell_2=1$. The function $J_1(x)$ takes the value $0$ if $x=0$ and infinity otherwise. Similarly, $J_2(x)=(13/3)x^2$ if $|x|\leq 2/3$ and infinity otherwise.
\end{exmp}

Although we have described the algorithm by assigning computational tasks to several distributed units, it should be clear that the scheme can also be used if there is one single computing unit available. In this case, the scheme is executed by obtaining $(T^{\ell_i}J_i)(x)$ in sequential computation. 

{In what follows, we will show the equivalence between parallel MPC and an MPC scheme involving a single lookahead minimization with a suitably defined terminal cost.} We will also analyze the properties of the corresponding policy $\tilde{\mu}$, and provide a variant of the scheme. For parallel MPC defined by \eqref{eq:rollout} and \eqref{eq:index_unit}, we denote by $\ell$ the minimum value of the lookahead steps, i.e., $\ell=\min_i\{\ell_i\}$. In addition, we define $\Bar{J}_i=T^{\ell_i-\ell}J_i$. As a result, the minimization \eqref{eq:rollout} involving different lookahead steps can be transformed to the ones that have \emph{a common lookahead step $\ell$} {with terminal costs $\Bar{J}_i$}.

\subsection{Justification of Parallel MPC}
In what follows, we show that parallel MPC is equivalent to a single lookahead minimization with some function that approximates $J^*$. For this purpose, let us introduce some concepts. For every function $J\in \mathcal{E}^+(X)$ and state $x\in X$, we define the set $\tilde{U}(J,x)$ as
\begin{equation}
    \label{eq:min_variable}
    \tilde{U}(J,x)=\arg\min_{u\in U(x)} \big\{g(x,u)+J\big(f(x,u)\big)\big\}.
\end{equation}
Assumption~\ref{assum:standing} asserts that these sets are always nonempty. In addition let $\{\Bar{J}_i\}_{i=1}^p\subset \mathcal{E}^+(X)$ be a finite collection of nonnegative functions, and $\Bar{J}$ be their pointwise minimum
\begin{equation}
    \label{eq:j_bar_def}
    \Bar{J}(x)=\min_{i=1,2,\dots,p}\Bar{J}_i(x),\,\forall x\in X.
\end{equation}
For every nonnegative integer $\ell$ and state $x$, we define the nonempty set of the indices of the computing units
\begin{equation}
    \label{eq:i_set_def}
    \tilde{I}(\ell,x)=\arg\min_{i=1,2,\dots,p}(T^\ell\Bar{J}_i)(x).
\end{equation}
Aided by those concepts, we provide our first result, which states that parallel MPC involving $\Bar{J}_i$ is equivalent to solve a single lookahead minimization with $\Bar{J}$ defined in \eqref{eq:j_bar_def}. This equivalence is shown both in terms of minimum value [part (a)], and also the control that attains the minimum [part (b)].

\begin{proposition}\label{prop:parellel}
Let $\{\Bar{J}_i\}_{i=1}^p\subset \mathcal{E}^+(X)$ where $p$ is a positive integer, $\Bar{J}$ be given by \eqref{eq:j_bar_def}, and $\ell$ be a positive integer.
\begin{itemize}
    \item[(a)] The equality 
    \begin{align*}
        &\min_{u\in U(x)} \big\{g(x,u)+(T^{\ell-1}\Bar{J})\big(f(x,u)\big)\big\}\\
    =&\min_{i=1,2,\dots,p}\bigg[\min_{u\in U(x)} \big\{g(x,u)+(T^{\ell-1}\Bar{J}_i)\big(f(x,u)\big)\big\}\bigg]
    \end{align*}
    holds for all $x\in X$, or equivalently,
    \begin{equation}
        \label{eq:paral_equiv}
    (T^\ell\Bar{J})(x)=\min_{i=1,2,\dots,p}\big[(T^\ell\Bar{J}_i)(x)\big],\,\forall x\in X.
    \end{equation}
    \item[(b)] The equality 
    \begin{equation}
        \label{eq:paral_equiv_u}
        \tilde{U}(T^{\ell-1}\Bar{J},x)=\cup_{i\in \tilde{I}(\ell,x)}\tilde{U}(T^{\ell-1}\Bar{J}_i,x)
    \end{equation}
    holds for all $x\in X$, where $\tilde{I}(\ell,x)$ is defined as in \eqref{eq:i_set_def}. 
    \end{itemize}

\end{proposition}
\begin{proof} We will first show that the proposition hold for $\ell=1$. The cases where $\ell>1$ are shown by induction.

(a) Suppose $\ell=1$. For every fixed $i$, by \eqref{eq:j_bar_def} we have that $\Bar{J}(x)\leq \Bar{J}_i(x)$ for all $x$. By the monotonicity property \eqref{eq:monotone}, using $\Bar{J}$ and $\Bar{J}_i$ in place of $J$ and $J'$, and taking minimum on both sides, we have 
\begin{align*}
    &\min_{u\in U(x)} \big\{g(x,u)+\Bar{J}\big(f(x,u)\big)\big\}\\
    \leq&\min_{u\in U(x)} \big\{g(x,u)+\Bar{J}_i\big(f(x,u)\big)\big\},\,\forall x\in X.
\end{align*}
By taking minimum over $i$ on the right-hand side, we have
\begin{align*}
    &\min_{u\in U(x)} \big\{g(x,u)+\Bar{J}\big(f(x,u)\big)\big\}\\
    \leq&\min_{i=1,2,\dots,p}\bigg[\min_{u\in U(x)} \big\{g(x,u)+\Bar{J}_i\big(f(x,u)\big)\big\}\bigg].
\end{align*}

For the converse, we note that for every $x$, $\tilde{U}(\Bar{J},x)$ defined in Eq.~\eqref{eq:min_variable} is nonempty by Assumption~\ref{assum:standing}. Let $\tilde{u}$ be an arbitrary element of $\tilde{U}(\Bar{J},x)$. Then 
$$\min_{u\in U(x)} \big\{g(x,u)+\Bar{J}\big(f(x,u)\big)\big\}=g(x,\tilde{u})+\Bar{J}\big(f(x,\tilde{u})\big).$$
In addition, if $\tilde{i}\in \tilde{I}\big(0,f(x,\tilde{u})\big)$ given by \eqref{eq:i_set_def},
$$g(x,\tilde{u})+\Bar{J}\big(f(x,\tilde{u})\big)=g(x,\tilde{u})+\Bar{J}_{\tilde{i}}\big(f(x,\tilde{u})\big).$$
However, we also have that 
\begin{align*}
    &\min_{i=1,2,\dots,p}\bigg[\min_{u\in U(x)} \big\{g(x,u)+\Bar{J}_i\big(f(x,u)\big)\big\}\bigg]\\
    \leq &g(x,\tilde{u})+\Bar{J}_{\tilde{i}}\big(f(x,\tilde{u})\big),
\end{align*}
which concludes the proof for $\ell=1$. In particular, we obtain the equality
\begin{equation}
\label{eq:paral_equiv_l1}
    \begin{aligned}
        &\min_{u\in U(x)} \big\{g(x,u)+\Bar{J}\big(f(x,u)\big)\big\}\\
    =&\min_{i=1,2,\dots,p}\bigg[\min_{u\in U(x)} \big\{g(x,u)+\Bar{J}_i\big(f(x,u)\big)\big\}\bigg]
    \end{aligned}
\end{equation}

Let us assume that \eqref{eq:paral_equiv} hold for $\ell-1$ with $\ell>1$, then 
    \begin{align*}
        &\min_{u\in U(x)} \big\{g(x,u)+(T^{\ell-1}\Bar{J})\big(f(x,u)\big)\big\}\\
        =&\min_{u\in U(x)} \Big\{g(x,u)+\min_{i=1,2,\dots,p}\big[(T^{\ell-1}\Bar{J}_i)\big(f(x,u)\big)\big]\Big\}.
    \end{align*}
    Applying the equality \eqref{eq:paral_equiv_l1} with $T^{\ell-1}\Bar{J}_i$ in place of $\Bar{J}_i$, and $\min_{i=1,2,\dots,p}(T^{\ell-1}\Bar{J}_i)(x)$ in place of $\Bar{J}(x)$, yields that
    \begin{align*}
        &\min_{u\in U(x)} \Big\{g(x,u)+\min_{i=1,2,\dots,p}\big[(T^{\ell-1}\Bar{J}_i)\big(f(x,u)\big)\big]\Big\}\\
        =&\min_{i=1,2,\dots,p}\bigg[\min_{u\in U(x)} \big\{g(x,u)+(T^{\ell-1}\Bar{J}_i)\big(f(x,u)\big)\big\}\bigg].
    \end{align*}
    Combining two equations gives the desired result.

(b) We show that the result holds for $\ell=1$. The induction part is similar to that of the proof for part (a) and is thus omitted. Given an $x$, let $\tilde{u}\in \tilde{U}(\Bar{J},x)$, then we have that
    $$(T\Bar{J})(x)=g(x,\tilde{u})+\Bar{J}\big(f(x,\tilde{u})\big).$$
    Also, if $\tilde{i}\in \tilde{I}\big(0,f(x,\tilde{u})\big)$, where $\tilde{I}(\cdot,\cdot)$ is given as in \eqref{eq:i_set_def}, then it yields that
    $$g(x,\tilde{u})+\Bar{J}\big(f(x,\tilde{u})\big)=g(x,\tilde{u})+\Bar{J}_{\tilde{i}}\big(f(x,\tilde{u})\big).$$
    By part (a), we have that
    $$g(x,\tilde{u})+\Bar{J}_{\tilde{i}}\big(f(x,\tilde{u})\big)=\min_{i=1,2,\dots,p}\big[(T\Bar{J}_i)(x)\big].$$
    Therefore, $\tilde{i}\in \tilde{I}(1,x)$ and $\tilde{u}\in \tilde{U}(\Bar{J}_{\tilde{i}},x)$.

    Conversely, let $\tilde{i}\in \tilde{I}(1,x)$ and $\tilde{u}\in \tilde{U}(\Bar{J}_{\tilde{i}},x)$. We will show that $(T\Bar{J})(x)=g(x,\tilde{u})+\Bar{J}\big(f(x,\tilde{u})\big)$. By the definition of $ \tilde{I}(1,x)$ and $\tilde{U}(\Bar{J}_{\tilde{i}},x)$, we have that $(T\Bar{J})(x)=g(x,\tilde{u})+\Bar{J}_{\tilde{i}}\big(f(x,\tilde{u})\big)$. On the other hand, we also have 
    $$(T\Bar{J})(x)\leq g(x,\tilde{u})+\Bar{J}\big(f(x,\tilde{u})\big)\leq g(x,\tilde{u})+\Bar{J}_{\tilde{i}}\big(f(x,\tilde{u})\big),$$
    which gives the desired result. 
\end{proof}
\begin{remark}\label{rmk:parellel}
When $\Bar{J}_i$ are the cost functions of certain policies, this result has been established in \cite{bertsekas1997rollout} for combinatorial problems. It has also been shown for the Markovian decision problem in \cite[Prop.~5.2.2]{bertsekas2020rollout}, the stochastic shortest path problem under certain conditions, as well as deterministic problems \cite[p.~378]{bertsekas2020rollout}. We generalize these results by allowing $\Bar{J}_i$ to be arbitrary nonnegative functions.
\end{remark}

In view of Prop.~\ref{prop:parellel}(a), we can interpret the proposed scheme \eqref{eq:bellman_p_rollout} as solving $(T^\ell \Bar{J})(x_k)$ for some integer $\ell$ and some function $\Bar{J}$. To see this, define {$\ell=\min_i \{\ell_i\}$}, and $\Bar{J}_i=T^{\ell_i-\ell}J_i$. Moreover, Prop.~\ref{prop:parellel}(b) establishes that the unit $\tilde{i}$ selected by the central unit with the smallest value $\tilde{J}_i(x_k)$ must be an element of $\tilde{I}(\ell,x_k)$. In addition, the first control $\tilde{u}_k^{\tilde{i}}$ must be an element of $\tilde{U}(T^{\ell-1}\Bar{J},x_k)$. Therefore, the proposed algorithm solves the problem $(T^\ell\Bar{J})(x_k)$ with $\ell$ and $\Bar{J}$ defined accordingly. As a result of Prop.~\ref{prop:parellel}, analysis on the parallel MPC can be transformed to studying $T^\ell\Bar{J}$. {Indeed, the left side of Eq.~\eqref{eq:paral_equiv} serves as an analytical device for our investigation on the properties of parallel MPC, while the right side of Eq.~\eqref{eq:paral_equiv} represents the actual computational approach we employ. This is because the optimization problem defined by $(T^\ell\Bar{J})(x)$ is often more challenging to solve compared with that defined by the right side of Eq.~(14); see Examples~\ref{eg:scalar_intro_central} and \ref{eg:lq}. Yet, it is easier to analyze the properties of the policy defined by $(T^\ell\Bar{J})(x)$.}

In what follows, we will show that $\Bar{J}$ inherits the desired properties of $J_i$, thus resulting in the desired behavior of the policy $\tilde{\mu}$.

\begin{proposition}\label{prop:parellel_rod}
Let $\{J_i\}_{i=1}^p\subset \mathcal{L}(X)$ where $p$ is a positive integer, and $\{\ell_i\}_{i=1}^p$ be a collection of $p$ positive integers. In addition, let $\ell=\min_i \{\ell_i\}$, $\Bar{J}_i=T^{\ell_i-\ell}J_i$, and $\Bar{J}$ be defined \eqref{eq:j_bar_def}. Then $\Bar{J}\in \mathcal{L}(X)$.
\end{proposition}

\begin{proof}
Note that if $J\in \mathcal{L}(X)$, then due to the monotonicity of $T$, we have that $T^j J\in  \mathcal{L}(X)$ for any nonnegative integer $j$. Therefore, $\Bar{J}_i=T^{\ell_i-\ell}J_i\in\mathcal{L}(X)$. This yields that 
    $$(T\Bar{J}_i)(x)\leq \Bar{J}_i(x),\,\text{for all $x$ and $i$}.$$
    Taking minimum over $i$, the left side becomes $(T\Bar{J})(x)$ due to \eqref{eq:paral_equiv} in Prop.~\ref{prop:parellel}(a) for $\ell=1$, and the right side equals $\Bar{J}(x)$ by \eqref{eq:j_bar_def}.
\end{proof}

Finally, we show that the minimum value computed in \eqref{eq:bellman_p_rollout} provides an upper bound for the cost function of the parallel MPC policy.

\begin{proposition}\label{prop:parellel_ro_bound}
The cost function $J_{\tilde{\mu}}(x)$ of the parallel MPC policy $\tilde{\mu}$ defined through \eqref{eq:rollout} and \eqref{eq:index_unit}, is upper bounded by the value $\tilde{J}(x)$ obtained via minimization \eqref{eq:bellman_p_rollout} by the central unit.
\end{proposition}
\begin{proof}
    Our preceding discussion has shown that the parallel MPC scheme is equivalent to computing $(T^\ell\Bar{J})(x)$. In addition, we have that for $i=1,\dots,p$,
    \begin{align*}
        &g\big(x,\tilde{\mu}(x)\big)+\big(T^{\ell-1}\Bar{J}\big)\big(f\big(x,\Tilde{\mu}(x)\big)\big)\\
        =& (T^\ell\Bar{J})(x)\leq (T^{\ell-1}\Bar{J})(x),
    \end{align*}
    where the last inequality follows as $\Bar{J}\in\mathcal{L}(X)$, by Prop.~\ref{prop:parellel_rod}. {Applying Prop.~\ref{prop:classic}(b) with $T^{\ell-1}\Bar{J}$ as $J$}, we have $J_{\Tilde{\mu}}(x)\leq T^{\ell-1}\Bar{J}(x)$. By the monotonicity property \eqref{eq:monotone}, we have 
\begin{align*}
    J_{\Tilde{\mu}}(x)=&g\big(x,\Tilde{\mu}(x)\big)+J_{\Tilde{\mu}}\big(f\big(x,\Tilde{\mu}(x)\big)\big)\\
    \leq& g\big(x,\Tilde{\mu}(x)\big)+\big(T^{\ell-1}\Bar{J}\big)\big(f\big(x,\Tilde{\mu}(x)\big)\big)=(T^\ell\Bar{J})(x),
\end{align*}
where the first equality is due to Prop.~\ref{prop:classic}(a), the last equality is due to the definition of $\Tilde{\mu}$. The proof is thus complete.
\end{proof}

\begin{remark}\label{rmk:performance}
{Let us denote by $\tilde{\mu}_i$ the MPC policy defined by $T^{\ell_i}J_i$; i.e., the MPC policy computed by the $i$th unit. Classical MPC theory states that $J_{\tilde{\mu}_i}(x)\leq \tilde J_i(x)\leq \infty$ for all $x\in X$. The proposition states that the cost function of the parallel MPC policy $\tilde{\mu}$ is upper bounded by the pointwise minimum of the functions $\tilde{J}_i$ computed via individual lookahead. Therefore, $J_{\tilde{\mu}}(x)<\infty$ as long as $\tilde J_i(x)<\infty$ for some $i$. This implies that the set of the states $x$ where $J_{\tilde{\mu}}(x)<\infty$ is the union of the sets $\{x\in X\,|\,J_{\tilde{\mu}_i}(x)<\infty\}$, $i=1,2,\dots,p$.}
\end{remark}

\begin{remark}\label{rmk:pi_imp}
{Under certain conditions, we can show that the cost function $J_{\tilde\mu}$ satisfies $J_{\tilde\mu}\leq J_{\tilde\mu_i}$ for some $i$. In particular, if $J_j=J_{\tilde\mu_i}$ for some $j\neq i$, then $\tilde J_j=T^{\ell_j}J_j\leq J_j=J_{\tilde\mu_i}$. Such a case arises when the function $J_j$ is constructed via collecting state-and-control trajectory data under $\tilde\mu_i$, as is proposed in \cite{rosolia2017learning}. Our computational study, reported in Section~\ref{sec:exmp}, also provides empirical evidence that the policy $\tilde{\mu}$ often outperforms the MPC policies defined by the individual lookahead minimization even without such conditions.}
\end{remark}

\begin{remark}\label{rmk:parellel_rollout}
If $\ell_i=\ell$ for all $i$, and the functions $J_i$ are cost functions of some policies $\mu_i$, then the result here recovers the conclusion stated in \cite[p.~378]{bertsekas2020rollout}.  
\end{remark}

{We demonstrate these propositions using our running example.

\begin{exmp}\label{eg:scalar_intro_central} Consider the parallel MPC applied to the problem studied in Examples~\ref{eg:scalar_intro} and \ref{eg:scalar_intro_re}. Prop.~\ref{prop:parellel} states that the minimum of the optimal values computed by the two units, as shown in Example~\ref{eg:scalar_intro_re}, equals to the minimum $\tilde J(x_k)$ of the following problem:
\begin{equation}
\label{eq:scalar_central}
    \begin{aligned}
    \min_{u_k\in U}& \quad  g(x_{k},u_{k})+\bar J(x_{k+1})\\
	\mathrm{s.\,t.} & \quad \bar J(x_{k+1})=\min\big\{\bar J_1(x_{k+1}),\bar J_2(x_{k+1})\big\},\\
    & \quad \bar J_1(x_{k+1})=\min_{\substack{u_{k+1}\in U\\2x_{k+1}+u_{k+1}=0}}g(x_{k+1},u_{k+1}),\\
	& \quad  \bar J_2(x_{k+1})=\begin{cases}13/3x_{k+1}^2,&\text{if }|x_{k+1}|\leq 2/3,\\
    \infty,&\text{otherwise;}\end{cases}
\end{aligned}
\end{equation}
cf. Eq.~\eqref{eq:paral_equiv}. Moreover, the control $\tilde u_k$ that attains the minimum in problem \eqref{eq:scalar_central} also attains the minimum of the problem computed by the unit with smaller minimum values in problems \eqref{eq:scalar_parallel}; cf. Eq.~\eqref{eq:paral_equiv_u}. Clearly, problems \eqref{eq:scalar_parallel} are easier to solve than problem \eqref{eq:scalar_central}. 

Taking as given that $\bar J_1,\bar J_2\in \mathcal{L}(X)$ (we will show this later), Props.~\ref{prop:parellel_rod} and \ref{prop:parellel_ro_bound} state that the policy $\tilde \mu$ defined by $\tilde{\mu}(x_k)=\tilde u_k$ satisfies $J_{\tilde\mu}(x_k)\leq \tilde J(x_k)$. 
\end{exmp}}

\subsection{Properties of the L-Region}
We now address the question on how to obtain $J_i$ that belongs to the L-region defined in \eqref{eq:rod}. Prop.~\ref{prop:classic}(a) suggests that any $J_\mu$ constitutes a valid option. The following result from \cite{li2021data} (included here for completeness) shows that $J_i$ can be used to encode simultaneously the Lyapunov inequality of the terminal cost and the invariance property of the terminal constraint. The cost and the constraint can be computed via optimization, as in \cite{li2023performance,johansson2024stable}, as well as via sampled data \cite{rosolia2017learning,li2021data}. To this end, we denote by $\delta_C:X\mapsto \{0,\infty\}$ the function taking value $0$ if $x\in C$ and infinity otherwise.

\begin{proposition}\label{prop:d_fun_policy}
Let $S\subset X$, $V\in \mathcal{E}^+(X)$, and $\mu$ be a policy such that
\begin{equation}
    \label{eq:invariance}
    x\in S\implies f\big(x,\mu(x)\big)\in S,
\end{equation}
and
\begin{equation}
    \label{eq:lyap_mpc}   g\big(x,\mu(x)\big)+V\Big(f\big(x,\mu(x)\big)\Big)\leq V(x),\quad \forall x\in S.
\end{equation}
Then $J\in \mathcal{E}^+(X)$ defined as $J(x)=V(x)+\delta_S(x)$ belongs to the L-region. 
\end{proposition}

\begin{remark}\label{rmk:d_fun_policy}
The condition stated in \eqref{eq:invariance} is often referred to as the invariance property of the set $S$, also known as $S$ being \emph{strong reachable}, introduced first in \cite{bertsekas1971control}. When the function $V$ takes nonnegative real numbers, the inequality \eqref{eq:lyap_mpc} is equivalent to the Lyapunov inequality. Both of these concepts are extensively studied in the MPC literature, and Prop.~\ref{prop:d_fun_policy} shows that they can be encoded as the property defining the L-region. This is convenient for our algorithmic design and analysis. However, we do not investigate the computation of the set $S$ and the function $V$ in the present work.   
\end{remark}

{We show next that the functions $\bar J_1$ and $\bar J_2$ in Example~\ref{eg:scalar_intro_central} belong to the L-region using Prop.~\ref{prop:d_fun_policy}.

\begin{exmp}\label{eg:scalar_intro_central} Let us consider $\bar J_1$ and $\bar J_2$ defined in the problem \eqref{eq:scalar_central}. We define $J_1=V_1+\delta_{S_1}$, where $V_1\equiv0$ and $S_1=\{0\}$. Let $\mu_1(0)=0$. Then one can verify that conditions~\eqref{eq:invariance} and \eqref{eq:lyap_mpc} are satisfied with $S_1$, $V_1$, and $\mu_1$ in place of $S$, $V$, and $\mu$. Therefore, $J_1$ belongs to the L-region. Due to monotonicity of $T$, we have $\bar J_1=TJ_1$ belong to the L-region as well. 

As for $\bar J_2$, we consider $\mu_2(x)=-3/2x$. In view of the control constraint $U=[-1,1]$, $\mu_2(x)=-3/2x$ is defined for $S_2=\{x\,|\,|x|\leq 2/3\}$. One can verify that $x_2\in S_2$ implies $f(x,\mu_2(x))\in S_2$. Moreover, for every $x\in S_2$, $J_{\mu_2}(x)=13/3x^2$. As a result, we set $V_2=J_{\mu_2}$ and $g(x,-3/2x)+V_2(1/2x)=V_2(x)$; cf. condition \eqref{eq:lyap_mpc}. Since $\bar J_2=V_2+\delta_{S_2}$, $\bar J_2$ belongs to the L-region.
\end{exmp}}

The following {result} show that for an element $J\in \mathcal{L}(X)$, the values of $(T^\ell J)(x_k)$ are decreasing as $k$ increases, provided that the state sequence $\{x_k\}$ is generated under the policy $\tilde{\mu}$ defined by
$$\tilde{\mu}(x)\in \arg\min_{u\in U(x)} \Big\{g(x,u)+(T^{\ell-1}J)\big(f(x,u)\big)\Big\}.$$  

\begin{proposition}\label{prop:parellel_lya}
Let $J \in\mathcal{L}(X)$, $\ell$ be a positive integer, and for a given $x$, let $\tilde{u}$ be an element in $\tilde{U}(T^{\ell-1}J,x)$. Then we have
$$(T^\ell J)\big(f(x,\tilde{u})\big)\leq (T^\ell J)(x).$$
\end{proposition}

\begin{proof}
    We first show the result hold for $\ell=1$. We will then establish the inequality for $\ell>1$. 
    
    By the definition of $\tilde{U}(J,x)$, we have that 
    $(TJ)(x)=g(x,\tilde{u})+J\big(f(x,\tilde{u})\big).$
    The function $J$ belongs to the L-region so that  
    $g(x,\tilde{u})+(TJ)\big(f(x,\tilde{u})\big)\leq g(x,\tilde{u})+J\big(f(x,\tilde{u})\big).$
    Given that $g(x,\tilde{u})\in [0,\infty]$, we obtain the equality
    \begin{equation}
    \label{eq:lyap_1}
        (TJ)\big(f(x,\tilde{u})\big)\leq (TJ)(x).
    \end{equation}

    For $\ell>1$, since $J\in \mathcal{L}(X)$, then in view of the monotonicity of $T$, we have $T^{\ell}J\in \mathcal{L}(X)$. Applying the inequality \eqref{eq:lyap_1} with $T^{\ell}J$ replacing $J$, we get the desired inequality.
\end{proof}

\subsection{A Simplified Variant of Parallel MPC}
For some problems, computing $(TJ)(x)$ could be challenging due to the size of the set $U(x)$. In such cases, it may be convenient to restrict our attention to some subsets of $U(x)$. Given $\Bar{U}(x)\subset U(x)$, we introduce the \emph{simplified Bellman operator} $\Bar{T}:\mathcal{E}^+(X)\mapsto\mathcal{E}^+(X)$, which is defined pointwise as
\begin{equation}
    \label{eq:sim_bellman_op}
    (\Bar{T}J)(x)=\inf_{u\in \Bar{U}(x)} \big\{g(x,u)+J\big(f(x,u)\big)\big\}.
\end{equation}
Similar to the Bellman operator, we denote as $\Bar{T}^\ell$ the $\ell$-fold composition of $\Bar{T}$ with itself $\ell$ times, with the convention that $\Bar{T}^0J=J$. It is straightforward to verify that $\Bar{T}$ also has the monotonicity property: $J(x)\leq J'(x)$ for all $x\in X$ implies that $(\Bar{T}J)(x)\leq (\Bar{T}J')(x)$ for all $x\in X$. In addition, for every $J\in \mathcal{E}^+(X)$, we have that $(TJ)(x)\leq (\Bar{T}J)(x)$ for all $x\in X$. For functions defined through the operator $\Bar{T}$, we have the following results.

\begin{proposition}\label{prop:simplified}
Let $J \in\mathcal{L}(X)$ and $\Bar{U}(x)\subset U(x)$ such that
$$(\Bar{T}J)(x)\leq J(x),\quad \forall x\in X,$$
where $\Bar{T}$ is given in \eqref{eq:sim_bellman_op}. Then we have that $\Bar{T}^{\ell}J\in \mathcal{L}(X)$ where $\ell$ is a positive integer.
\end{proposition}

\begin{proof}
    Let us first show the inequality for $\ell=1$. Since $(\Bar{T}J)(x)\leq J(x)$ for all $x$, in view of the monotonicity property of $\Bar{T}$, we have that 
    $$(\Bar{T}^2J)(x)=\big(\Bar{T}(\Bar{T}J)\big)(x)\leq (\Bar{T}J)(x),\quad \forall x\in X.$$
    On the other hand, 
    $$\big(T(\Bar{T}J)\big)(x)\leq \big(\Bar{T}(\Bar{T}J)\big)(x),\quad \forall x\in X.$$
    Combining the above inequalities, we have that
    $$\big(T(\Bar{T}J)\big)(x)\leq (\Bar{T}J)(x),\quad \forall x\in X,$$
    which means that $\Bar{T}J\in \mathcal{L}(X)$.

    By recursively applying above arguments with $\Bar{T}^{\ell-1}J$, $\ell=1,2,\dots$, in place of $J$, we obtain the desired result.
\end{proof}

In view of the above result, we can introduce a simplified parallel MPC given as follows: instead of solving for each unit $(T^{\ell_i}J_i)(x)$, we may instead solve $\big(T(\Bar{T}_i^{\ell_i-1}J_i)\big)(x)$, where $\Bar{T}_i$ are simplified Bellman operators that can vary with $i$. Conceptually, we can view $\Bar{T}^{\ell_i-1}J_i$ as a function computed to approximate $J^*$ while each unit solves a one-step lookahead problem. 

\section{Illustrating Examples}\label{sec:exmp}
In this section, we illustrate through examples the validity of the parallel MPC applied to problems with deterministic dynamics. We aim to elucidate the following points:
\begin{itemize}
    \item[(1)] Parallel MPC can be applied to a wide range of optimal control problems with deterministic dynamics.
    \item[(2)] Parallel implementation improves computational efficiency at no cost of optimality (Example~\ref{eg:lq}); see Prop.~\ref{prop:parellel}.
    \item[(3)] Parallel MPC provides better performance guarantees than that of the individual lookahead minimization involved in the method (Example~\ref{eg:pwa}); see Prop.~\ref{prop:parellel_ro_bound}.
    \item[(4)] The simplified variant is well-suited for problems with discrete components  (Example~\ref{eg:switch}); see Prop.~\ref{prop:simplified}.
\end{itemize}
We use MATLAB with toolbox MPT \cite{kvasnica2004multi} to solve Examples~\ref{eg:lq}, \ref{eg:pwa}, and \ref{eg:switch}, and we also use toolbox YALMIP \cite{lofberg2004yalmip} in Example~\ref{eg:switch}. The code in MATLAB to reproduce all the examples and figures in this section, as well as additional examples, can be found in \url{https://github.com/akarapet/pr-mpc}. Computational details of $J_i$ used in the examples are peripheral to the parallel MPC ideas and are thus deferred towards the appendices.

\begin{exmp}\label{eg:lq}(Linear Quadratic Problem with Constraints) In this example, we consider a two-dimensional linear quadratic (LQ) problem with constraint. We aim to show that parallel MPC can harness the parallel computing units at no cost of optimality. Here the state space is $X=\Re^2$, where $\Re^n$ denotes the $n$-dimensional Euclidean space, and the control constraints are $U(x)=U=\Re$. The system dynamics are given by $f(x,u)=Ax+Bu$, where 
$$A = \begin{bmatrix}
        1 & 1\\
        0 & 1\end{bmatrix},\qquad B=\begin{bmatrix}
        1\\
        0.5\end{bmatrix}.$$
The stage cost is given by $g(x,u)=x'Qx +Ru^2+\delta_C(x,u)$, where prime denotes transposition, $Q$ is an identity matrix, and $R=1$. The set $C$ is $\{(x,u)\,|\,\Vert x\Vert_\infty\leq 5,\,|u|\leq 1\}$ ($\Vert x\Vert_\infty$ denotes the infinity norm of the vector $x$). We assume that four computational units are available, where $\ell_i=2$ for $i=1,2$ and $\ell_i=3$ for $i=3,4$. The functions $J_i(x)$ are computed in the form of $x'K_{i}x+\delta_{S_i}(x)$ so that $J_i\in \mathcal{L}(X)$ by Prop.~\ref{prop:d_fun_policy}; see Appendix~A.

{As discussed after Prop.~\ref{prop:parellel}, parallel MPC defined through Eqs.~\eqref{eq:rollout}-\eqref{eq:bellman_p_rollout} can be described succinctly as the right-hand side of Eq.~\eqref{eq:paral_equiv}. Prop.~\ref{prop:parellel} shows that parallel MPC is equivalent to solve directly the problem defined by left-hand side of Eq.~\eqref{eq:paral_equiv}, which we refer to as alternative approach from now on. Compared with the parallel MPC, this alternative approach uses one computing unit and solves a single lookahead minimization problem involving binary variables. Our computational studies show that the parallel MPC significantly improves the computational efficiency compared with the alternative approach. For example, Fig.~\ref{fig:comp_time} shows the computational time required for parallel MPC and alternative approach for the states $\{x_k\}$ of trajectories starting from $x_0 = [-5~2.7]'$ and $[2.3~-0.6]'$ and generated under the parallel MPC policy. We repeated the test with $100$ trials and averaged over the trials. We note that on average in this experiment the time required for parallel MPC is about $15\%$ of that needed for the alternative approach.}

\begin{figure}[ht!]
    \centering
    \includegraphics[width=0.7\linewidth]{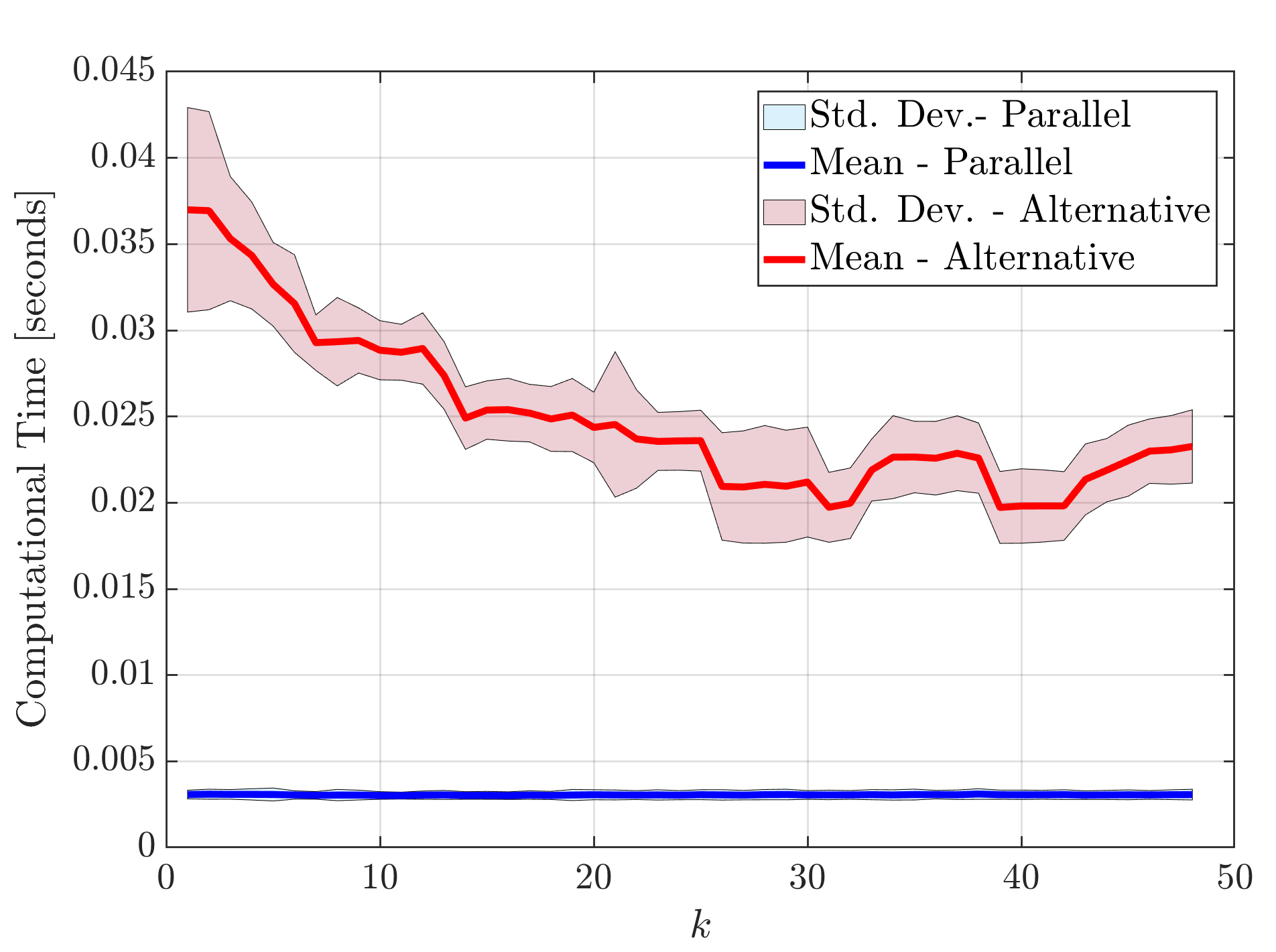}
    \caption{The computational time of parallel MPC (labeled as Parallel in the figure) is compared to that of the alternative approach with binary variables. The horizontal axis represents the index of the state along the state trajectories $\{x_k\}$, which are generated under the parallel MPC policy $\tilde{\mu}$. The computational time is averaged over $100$ trials.}
    \label{fig:comp_time}
\end{figure}
\end{exmp}

\begin{exmp}\label{eg:pwa}(Discontinuous Dynamics) In this example we demonstrate that the cost $J_{\tilde{\mu}}(x)$ of the policy $\tilde{\mu}$ computed via parallel MPC is upper bounded by the value $(T^\ell \Bar{J})(x)$. This bound is no worse than the bounds $(T^{\ell_i}J_i)(x)$, which are computed via the individual lookahead minimization. For this purpose, we consider the problem studied in \cite[Section~VI]{lazar2004stabilization}, where the dynamics are  
$$x_{k+1}=\begin{cases}A_1x_k+Bu_k,&\text{if }[1\; 0]x_k\geq 0,\\
    A_2x_k+Bu_k,&\text{otherwise.}
\end{cases}$$
The parameters of the system dynamics are
\begin{equation*}
    A_1 = \begin{bmatrix}
        0.35 & -0.61\\
        0.61 & 0.35\end{bmatrix},\;A_2 = \begin{bmatrix}
        0.35 & 0.61\\
        -0.61 & 0.35\end{bmatrix},\;
        B=\begin{bmatrix}
        0\\
        1\end{bmatrix}.
\end{equation*}
The stage cost is $g(x,u)=x'Qx +Ru^2+\delta_C(x,u)$, where $R=0.4$, and the matrix $Q$ and the set $C$ are the same as in Example~\ref{eg:lq}. We assume four computational units and set $\ell_i=2$ for all $i$, and introduce for each unit a function $J_i\in \mathcal{L}(X)$; see Appendix~B for the computational details. Table~\ref{table:hybrid} lists the results where parallel MPC is applied to some states $x_0$. These include the bounds $\tilde{J}_i$ computed by the individual units, the unit $\tilde{i}$ selected by the central unit, the actual performance bound $(T^\ell \Bar{J})(x_0)$, which bounds the cost function values $J_{\tilde{\mu}}(x_0)$; see Prop.~\ref{prop:parellel_ro_bound}. Empirically, we also found that the cost function $J_{\tilde{\mu}}$ of parallel MPC policy is lower than that of the policies defined by the individual lookahead minimization involved in the scheme. In other words, parallel MPC outperforms the MPC policies defined by the individual lookahead minimization.         
\begin{table}[ht!]
\caption{Optimal values and costs in Example~\ref{eg:pwa}.}
\label{table:hybrid}
\centering
\begin{tabular}{ |c|c|c|c|c|c|c|c| } 
\hline
 $x_0$ & $\tilde{J}_1$ & $\tilde{J}_2$ & $\tilde{J}_3$ & $\tilde{J}_4$ & $\tilde{i}$ & $T^\ell \Bar{J}$ & $J_{\Tilde{\mu}}$ \\
\hline
$[-4\, 4.6]'$ & $53.4$ & $\infty$ & $53.4$ & $54.3$ & $3$ & $53.4$ & $52.2$\\
\hline
$[5\, 3.4]'$ & $53.1$ & $\infty$ & $\infty$ & $54.7$ & $1$ & $53.1$ & $52.0$ \\
\hline
$[-4.5\, 2.7]'$ & $39.2$ & $39.2$ & $39.1$ & $40.4$ & $3$ & $39.1$ & $38.4$ \\
\hline
$[-5\, -5]'$ & $\infty$ & $\infty$ & $91.1$ & $85.9$ & $4$ & $85.9$ & $81.4$ \\
\hline
\end{tabular}

\end{table}
\end{exmp}

\begin{exmp}\label{eg:switch}(Switching System) In this example we show that the simplified variant of parallel MPC is well-suited for problems involving switching systems, where the control has both discrete and continuous elements. We consider a problem modified from \cite[Section~IV.A]{zhang2011infinite}, where the dynamics are
$x_{k+1}=A_{d_k}x_k+B_{d_k}v_k,$ where $u=(v,d)$ with $v\in \Re$ and $d\in\{1,2\}$, and the control constraint set is $U(x)= \Re\times\{1,2\}$ for all $x$. In addition,
$$A_1 = \begin{bmatrix}
        2 & 1\\
        0 & 1\end{bmatrix},\;B_1=\begin{bmatrix}
        1\\
        1\end{bmatrix},\;A_2 = \begin{bmatrix}
        2 & 1\\
        0 & 0.5\end{bmatrix},\;B_2=\begin{bmatrix}
        1\\
        2\end{bmatrix}.$$
The stage cost is $g(x,u)=x'Qx +Rv^2+\delta_C(x,v)$, where the parameters $Q$, $R$, and $C$ are the same as those in Example~\ref{eg:lq}. Compared with the problem studied in \cite[Section~IV.A]{zhang2011infinite}, we impose state and control constraints; in particular the term $\delta_C(x,v)$ is added to the stage cost. We use two computing units and we set $\ell_i=5$ for both units. We also introduce control constraint subsets $\Bar{U}_i(x)=\Re\times \{i\}$, $i=1,2$. The details of $J_i$ are given in  Appendix~C. The results for simplified parallel MPC, computed for some states $x_0$, are summarized in Table \ref{table:switch}. The cost function values are compared against a lower bound of $J^*$, which is given by $T^8\Bar{J}_0$, where $\Bar{J}_0(x)\equiv0$ for all $x$. The difference between $J_{\tilde{\mu}}$ and $T^8\Bar{J}_0$ defines an optimality gap of $J_{\tilde{\mu}}$. The values of the gap indicate that the policy $\tilde{\mu}$ is near optimal. We note that the parallel MPC scheme solves $4$ quadratic programs per time step, whereas the lower bound solves $2^8$ quadratic programs of a similar kind but larger size. For the computation of the lower bound, see Appendix~D. Moreover, we also verified in these tests that the cost function values of parallel MPC policy are lower than that of the MPC policies defined by the individual lookahead. 

\begin{table}[ht!]
\caption{Optimal values and costs in Example~\ref{eg:switch}.}
\label{table:switch}
\centering
\begin{tabular}{ |c|c|c|c| } 
\hline
 $x_0$ &  $T^\ell \Bar{J}$ & $J_{\Tilde{\mu}}$ & $\frac{J_{\Tilde{\mu}}-T^8\overline{J}_0}{J_{\Tilde{\mu}}}$ \\
\hline
$[-4\, 4.6]'$ & $65.8$ & $65.8$ & ${<10^{-5}}$ \\
\hline
$[1.2\, 1.5]'$ & $89.2$ & $86.6$ & $<10^{-3}$ \\
\hline
$[-3.5\, 2.0]'$ & $123.3$ & $113.5$ & $<10^{-4}$ \\
\hline
$[-1.5\, -0.5]'$ & $34.6$ & $34.6$ & ${<10^{-4}}$  \\
\hline
\end{tabular}

\end{table}
\end{exmp}

    

\section{Conclusions}\label{sec:con}
We considered infinite horizon optimal control problems with deterministic dynamics. Drawing inspiration from parallel rollout, we proposed a parallel MPC method that computes approximate solutions. {The proposed method solves} multiple lookahead minimization problems with different lookahead steps, and terminal costs and terminal constraints. {We showed that the proposed method can harness the power of multiple computing units.} In addition, we {proved} that parallel MPC provides better performance bounds than that given by the individual lookahead minimization involved in the method. Our computational examples {suggest} that in addition to better bounds, parallel MPC {often} provides better performance than each of the MPC schemes that correspond to the individual lookahead minimizations. Preliminary tests show that the computational savings carry to the larger scale problems.

\appendix
\section*{Appendix}

In the appendix, we provide computational details of the functions $J_i$ used in Examples~\ref{eg:lq}, \ref{eg:pwa}, and \ref{eg:switch}. Using Prop.~\ref{prop:d_fun_policy}, all functions $J_i$ of the examples are constructed in the form of $V+\delta_S$. All these functions are computed offline before the realtime operation starts. We also provide a lower bound of $J^*$ obtained via value iteration. This bound is used in Example~\ref{eg:switch} to demonstrate the suboptimality of the proposed scheme. Note that alternative approaches for computing $J_i$ introduced in \cite{li2023performance} can also be applied here.

\section{Off-Line Computation for Example~\ref{eg:lq}}
In this example, we compute the functions $J_i$ such that for some policy $\mu_i$, $J_i(x)=J_{\mu_i}(x)$ for all $x$ in some set $S_i$. For $i=1$, we compute $L_1=-(B'K_{1}B+R)^{-1}B'K_{1}A$, where $K_{1}$ is the solution of the matrix equation 
$$K=A'(K-KB(B'KB+R)^{-1}B'K)A+Q.$$
The set $S_1$ is computed by the command \texttt{invariantSet()} in the MPT toolbox for the system $A+BL_1$ with the presence of the constraint set $C$. The function $J_1$ is defined as $x'K_{1}x+\delta_{S_1}(x)$. As a result, we have $J_1(x)=J_{\mu_1}(x)$ and $\mu_1(x)=L_1x$ for all $x\in S_1$. Similar to $i=1$, we compute $L_2=-[0.1\; 1.2]$, $L_3=-[0.2\; 0.7]$ and $L_4=-[0.3\; 0.8]$ for the units $i=2,3,4$. The values of $L_i$ for these units are obtained via trial and error. The $J_i$ for $i=2,3,4$ are defined as $x'K_{i}x+\delta_{S_i}(x)$, where $K_{i}$ is the solution to the equation
$$K=(A+BL_i)'K(A+BL_i)+Q+L_i'RL_i,\quad i=2,3,4.$$
The sets $S_2$ and $S_3$ are computed in the same way as is for $S_1$, while the set $S_4$ is defined as $S_4=\{x\in \Re^2\,|\,x'K_{4}x\leq \alpha^*\}$. The value $\alpha^*$ used for defining $S_4$ is the optimal value of the optimization problem
$$\min_{\alpha\in\Re}\  \alpha, \ \
	\mathrm{s.\,t.} \   x'K_{4}x\leq \alpha,\;\forall x \text{ so that }(x,L_4x)\in C.$$
This approach of constructing the set dates back to the thesis \cite{bertsekas1971control}. The constraint in the optimization can be transformed as linear constraints imposed on $\alpha$ via standard multiplier method, cf. \cite{bertsekas2014constrained}.

\section{Off-Line Computation for Example~\ref{eg:pwa}}\label{app:eg4_6}
In this example, we compute the functions $J_i$ such that for some policy $\mu_i$, we have that  $g\big(x,\mu_i(x)\big)+J_i\big(f\big(x,\mu_i(x)\big)\big)\leq J_i(x)$
for all $x$ in some set $S_i$. The policies $\mu_i$ for $i=1,2,3$ take the form $\mu_i(x)=L_i^1x$ if $[1\; 0]x_k\geq 0$ and $L_i^2x$ otherwise. The policy $\mu_1$ is the same as the one given in \cite[Section VI]{lazar2004stabilization}. The function $J_1$ takes the form $x'K_1x+\delta_{S_1}(x)$, with $K_1$ and $S_1$ given by \cite[Section VI]{lazar2004stabilization}. For the second unit, we set $L_2^i=-(B'K_i B+R)^{-1}B'K_iA_i$, where $K_i$ is the solution of
$$K=A_i'(K-KB(B'KB+R)^{-1}B'K)A_i+Q,\quad i=1,2.$$
We follow the procedure outlined in \cite[Sections IV, V]{lazar2004stabilization} and obtain
$$K_2=\begin{bmatrix}
    1.4786  &  0\\
   0 &    1.9605
    \end{bmatrix},$$
as well as the set $S_2=\{x\in\Re^2\,|\,P_2x\leq e_4\}$, where $e^n\in \Re^n$ is an $n$-dimensional vector with all elements equal to $1$, and $P_2$ is
$$P_2'=\begin{bmatrix}
    -0.4452  &  0.4452  &  0.4452  & -0.4452\\
   -0.3354  &  0.3354 &  -0.3354  &  0.3354
    \end{bmatrix}.$$

For the third unit, the parameters are synthesized by solving a linear matrix inequality given in \cite[Sections IV]{lazar2004stabilization}. The obtained parameters are $L_3^1=-[0.5112   \;0.2963]$ and $ L_3^2=[0.5112 \;-0.2963]$. Similarly, the function $J_3$ takes the form $x'K_3x+\delta_{S_3}(x)$ with
$$K_3=\begin{bmatrix}
    1.7847   &  0\\
   0 &    2.1375
    \end{bmatrix},$$
and the set $S_3=\{x\in\Re^2\,|\,P_3x\leq e_8\}$, where 
    \begin{equation*}
    P_3'=\begin{bmatrix}
    0.15 &  -0.15 &  -0.15  &  0.15  & -0.51  &  0.51   & 0.51  & -0.51\\
   -0.33 &   0.33  & -0.33 &   0.33 &  -0.30  &  0.30  & -0.30  &  0.30
    \end{bmatrix}.
\end{equation*}
Given that matrices $A_1$ and $A_2$ are both stable, we set $\mu_4(x)=0$ for all $x$. The function $J_4$ is given by $x'K_4x+\delta_{S_4}(x)$, where $S_4=\{x\in\Re^2\,|\,P_4x\leq e_4\}$, and
$$K_4=\begin{bmatrix}
    1.9631   &  0\\
   0 &    1.9631
    \end{bmatrix}, \ P_4'=\begin{bmatrix}
    1 &    0   & -1    & 0\\
     0  &   1   &  0  &  -1
    \end{bmatrix}.$$

\section{Off-Line Computation for Example~\ref{eg:switch}}
We define the functions $J_i$ such that for some policy $\mu_i$, $J_i(x)=J_{\mu_i}(x)$ for all $x$ in some set $S_i$. To this end, we set $L_i=-(B_i'K_{i}B_i+R)^{-1}B_i'K_{i}A_i$, $i=1,2$, where $K_{i}$ is the solution of the matrix equation
$$K=A_i'(K-KB_i(B_i'KB_i+R)^{-1}B_i'K)A_i+Q.$$ The set $S_i$ is computed by the command \texttt{invariantSet()} in the MPT toolbox for the system $A_i+B_iL_i$ with the presence of the constraint set $C$. The function $J_i$ is defined as $x'K_{i}x+\delta_{S_i}(x)$. This is the cost function of the policy that sets $d=i$, and $v=L_ix$ for $x\in S_i$. When applying the simplified variant at $x_k$ with $\ell_i=5$ and $J_i$, the computation on the $i$th unit takes the form $\big(T(\Bar{T}^4_iJ_i)\big)(x_k)$, where 
$(\Bar{T}_iJ)(x)=\inf_{u\in \Bar{U}_i(x)} \big\{g(x,u)+J\big(f(x,u)\big)\big\},$
with $\Bar{U}_i(x)=\Re\times \{i\}$, $i=1,2$; cf. \eqref{eq:sim_bellman_op}. This is equivalent to the following optimization problem:
\begin{subequations}
    \begin{align*}
	\min_{\{v_{k+j}\}_{j=0}^{4},d}& \quad x_{k+5}'K_ix_{k+5}+\sum_{j=0}^{4} x_{k+j}'Qx_{k+j}+Rv_{k+j}^2\\
	\mathrm{s.\,t.} & \quad x_{k+1}=A_dx_k+B_dv_k,\,d\in\{1,2\},\\
    & \quad x_{k+j+1}=A_ix_{k+j}+B_iv_{k+j},\,j = 1,...,4,\\
    & \quad  \Vert x_{k+j}\Vert_\infty\leq 5, \ j = 0,...,5, \\
	& \quad  |v_{k+j}|\leq 1, \ j = 0,...,4,\, x_{k+5}\in S_i.
	\end{align*}
\end{subequations} 

\section{A Lower Bound of Optimal Cost}\label{app:lb}

In Example~\ref{eg:switch}, the lower bound of $J^*(x_k)$ is $(T^8\Bar{J}_0)(x_k)$, which is the optimal value of the following optimization problem:
\begin{subequations}
    \begin{align*}
	\min_{\{(v_{k+j},d_{k+j})\}_{j=0}^{7}}& \quad \sum_{j=0}^{7} x_{k+j}'Qx_{k+j}+Rv_{k+j}^2\\
	\mathrm{s.\,t.} & \quad x_{k+j+1}=A_{d_k}x_{k+j}+B_{d_k}v_{k+j},\,j = 0,...,7,\\
    & \quad  \Vert x_{k+j}\Vert_\infty\leq 5, \ j = 0,...,7, \\
	& \quad  |v_{k+j}|\leq 1,\,d_{k+j}\in \{1,2\} \ j = 0,...,7.
	\end{align*}
\end{subequations}

\section*{Acknowledgement}
The authors would like to express their gratitude to Prof. Dimitri Bertsekas for the extensive helpful comments.

\bibliographystyle{plain}
\bibliography{ref}

\end{document}